\documentclass[journal]{IEEEtran}

\usepackage{graphics} 
\usepackage{epsfig} 
\usepackage{times} 
\usepackage{amsmath} 
\usepackage{amssymb}  
\usepackage{cite}
\usepackage{subfig}
\usepackage{multirow}
\newtheorem{proposition}{Proposition}[section]

\newtheorem{lemma}{Lemma}[section]
\newtheorem{theorem}{Theorem}[section]
\newtheorem{remark}{Remark}[section]
\newtheorem{definition}{Definition}

\title{On the Exact Solution to a \linebreak Smart Grid Cyber-Security Analysis Problem}

\author{Kin Cheong Sou, Henrik Sandberg and Karl Henrik Johansson\thanks{The authors are with the ACCESS Linnaeus Center and the Automatic Control Lab, the School of Electrical Engineering, KTH Royal Institute of Technology, Sweden. {\tt \small \{sou,hsan,kallej\}@kth.se} \newline This work is supported by the European Commission through the VIKING project, the Swedish Research Council (VR) under Grant 2007-6350 and Grant 2009-4565 and the Knut and Alice Wallenberg Foundation.}}

\begin{document}
\maketitle

\begin{abstract}
  This paper considers a smart grid cyber-security problem analyzing the vulnerabilities of electric power networks to false data attacks. The analysis problem is related to a constrained cardinality minimization problem. The main result shows that an $l_1$ relaxation technique provides an exact optimal solution to this cardinality minimization problem. The proposed result is based on a polyhedral combinatorics argument. It is different from well-known results based on mutual coherence and restricted isometry property. The results are illustrated on benchmarks including the IEEE 118-bus and 300-bus systems.
\end{abstract}

\begin{IEEEkeywords}
  Power network state estimation, security, operation research, optimization methods.
\end{IEEEkeywords}


\section{Introduction}
A modern society relies critically on the proper operation
of the electric power distribution and transmission system, which is supervised and controlled through Supervisory Control And Data Acquisition (SCADA) systems. Through remote terminal units (RTUs), SCADA systems measure data such as transmission line power flows, bus power injections and part of the bus voltages, and send them to the state estimator to estimate the power network states (e.g., the bus voltage phase angles and bus voltage magnitudes). The estimated states are used for vital power network operations such as optimal power flow (OPF) dispatch and contingency analysis (CA) \cite{Abur_Exposito_SEbook,Monticelli_SEbook}. See Fig.~\ref{fig:SCADA-attack} for a block diagram of the above functionalities.
\begin{figure}
\centering
\centering{\includegraphics[scale=0.4]{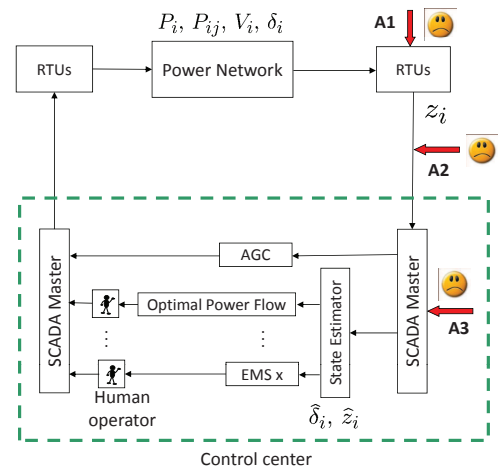}}
\caption{Block diagram of power network control center and SCADA. RTUs connected to the substations transmit and receive data from the control center using the SCADA system. At the control center, a state estimate is computed and then used by Energy Management Systems (EMS) to send out commands to the power network.
The human figures indicate where a human is needed in the control loop. This paper considers the false data attack scenario in A2.}\label{fig:SCADA-attack}
\end{figure}
Any malfunctioning of these operations can delay proper reactions in the control center, and lead to significant social and economical consequences such as the northeast US blackout of 2003.

The technology and the use of the SCADA systems have evolved a lot since the 1970's when they were introduced. The SCADA systems now are interconnected to office LANs, and through them they are connected to the Internet. Hence, today there are more access points to the SCADA systems, and also more functionalities to tamper with. For example, the RTUs can be subjected to denial-of-service attacks (A1 in Fig.~\ref{fig:SCADA-attack}). The communicated data can be subjected to false data attacks (A2). Furthermore, the SCADA master itself can be attacked (A3). This paper focuses on the cyber security issue related to false data attacks (A2), where the communicated metered measurements are subjected to additive data attacks. A false data attack can potentially lead to erroneous state estimates by the state estimator, which can result in gross errors in OPF dispatch and CA. In turn, these can lead to disasters of significant social and economical consequences. False data attack on communicated metered measurements has been considered in the literature (e.g.,\ \cite{LRN09,STJ_SCS2010,DS_SGC2010,BRWKNO10,KJTT11,SSJ_CDC2011_mincut,GBMKP11,KP11}). \cite{LRN09} was the first to point out that a coordinated intentional data attack can be staged without being detected by state estimation bad data detection (BDD) algorithm, which is a standard part of today's SCADA/EMS system. \cite{LRN09,STJ_SCS2010,DS_SGC2010,KJTT11,SSJ_CDC2011_mincut,GBMKP11,KP11} investigate the construction problem for such ``unobservable'' data attack, especially the sparse ones involving relatively few meters to compromise, under various assumptions of the network (e.g., DC power flow model \cite{Abur_Exposito_SEbook,Monticelli_SEbook}). In particular, \cite{LRN09} poses the attack construction problem as a cardinality minimization problem to find the sparsest attack including a given set of target measurements. \cite{STJ_SCS2010,DS_SGC2010,SSJ_CDC2011_mincut} set up similar optimization problems for the sparsest attack including a given measurement. \cite{KJTT11,KP11} seek the sparsest nonzero attack and \cite{GBMKP11} finds the sparsest attack including exactly two injection measurements. The solution information of the above optimization problems can help network operators identify the vulnerabilities in the network and strategically assign protection resources (e.g., encryption of meter measurements) to their best effect (e.g., \cite{DS_SGC2010,BRWKNO10,KP11}). On the other hand, the unobservable data attack problem has its connection to another vital EMS functionality, namely observability analysis \cite{Abur_Exposito_SEbook,Monticelli_SEbook}. In particular, solving the attack construction problem can also solve an observability analysis problem (this is to be explained in Section~\ref{subsec:msra}). This connection was first reported in \cite{KJTT11}, and was utilized in \cite{SSJ_ckt} to compute the sparsest critical $p$-tuples for some integer $p$. This is a generalization of critical measurements and critical sets \cite{Abur_Exposito_SEbook}.

To perform the cyber-security analysis in a timely manner, it is important to solve the data attack construction problem efficiently. This effort has been discussed, for instance, in \cite{LRN09,STJ_SCS2010,DS_SGC2010,KJTT11,SSJ_CDC2011_mincut,GBMKP11,KP11}. The efficient solution to the attack construction problem in \cite{STJ_SCS2010} is the focus of this paper. The matching pursuit method \cite{Mallat93matchingpursuit} employed in \cite{LRN09} and the basis pursuit method \cite{CDS_Basis_Pursuit} ($l_1$ relaxation and its weighted variant) employed in \cite{KP11} are common efficient (i.e., polynomial-time) approaches to suboptimally solve the attack construction problem. However, these methods do not guarantee exact optimal solutions, and in some cases they might not be sufficient (see for instance \cite{SSJ_CDC2011_mincut} for a naive application of basis pursuit and its consequences). While \cite{KJTT11,KP11} provide polynomial-time solution procedures for their respective attack construction problems, the problems therein are different from the one in this paper. Furthermore, the considered problem in this paper cannot be solved as a special case of \cite{KJTT11,KP11}. In particular, in \cite{KJTT11} the attack vector contains at least one nonzero entry. However, this nonzero entry cannot be given a priori. \cite{KP11} needs to restrict the number of nonzero injection measurements attacked, while there is no such requirement in the problem considered in this paper. In \cite{STJ_SCS2010,DS_SGC2010} a simple heuristics is provided to find suboptimal solutions to the attack construction problem. This heuristics, however, might not be sufficiently accurate. \cite{SSJ_CDC2011_mincut,SSJ_ckt,HJJSS12} is most closely related to the current work. The distinctions will be elaborated in Section~\ref{subsec:relationship_mincut}.

The main conclusion of this paper is that basis pursuit (i.e., $l_1$ relaxation) can indeed solve the data attack construction problem \emph{exactly}, under the assumption on the network metering system that no injection measurements are metered. The limitations of this assumption will be discussed in Section~\ref{subsec:H_assumptions}. In fact, the main result identifies a class of cardinality minimization problems where basis pursuit can provide exact optimal solutions. This class of problems include as a special case the considered data attack construction problem, under the assumption above. 

{\bf Outline:} Section~\ref{sec:model_opt_problems} describes the state estimation model and introduces the optimization problems considered in this paper. Section~\ref{sec:main_contribution} describes the main results of this paper -- the solution to the considered optimization problems. Section~\ref{sec:related_work} compares the proposed result to related works. Section~\ref{sec:proof} provides the proof the proposed main results. Section~\ref{sec:numerical_demonstration} numerically demonstrates the advantages of the proposed results.

\section{State Estimation and Cyber-Security Analysis Optimization Problems} \label{sec:model_opt_problems}

\subsection{Power network model and state estimation}
A power network with $n+1$ buses and $m_a$ transmission lines can be described as a graph with $n+1$ nodes and $m_a$ edges. The graph topology can be specified by the (directed) incidence matrix $B_0 \in \mathbb{R}^{(n+1) \times m_a}$, in which the direction of the edges can be assigned arbitrarily. The physical property of the network is described by a nonsingular diagonal matrix $D \in \mathbb{R}^{m_a \times m_a}$, whose nonzero entries are the reciprocals of the reactance of the transmission lines.

The states of the network include bus voltage phase angles and bus voltage magnitudes, the latter of which are typically assumed to be constant (equal to one in the per unit system). In addition, since one arbitrary bus is assigned as the reference with zero voltage phase angle, the network states considered can be captured in a vector $\theta \in {[0,2\pi)}^n$. The state estimator estimates the states $\theta$ based on the measurements obtained from the network. Under the DC power flow model \cite{Abur_Exposito_SEbook,Monticelli_SEbook} the measurement vector, denoted as $z$, is related to $\theta$ by
\begin{equation} \label{def:H_matrix}
  z = H \theta + \Delta z, \quad {\rm where} \quad H \triangleq
    \begin{bmatrix} P D B^T \\ Q B D B^T \end{bmatrix}.
\end{equation}
In (\ref{def:H_matrix}), $\Delta z$ can be either a vector of random error or intentional additive data attack (e.g., \cite{LRN09}). $B \in \mathbb{R}^{n \times m_a}$ is the truncated incidence matrix (i.e., $B_0$ with the row corresponding to the reference node removed), and $P$ consists of a subset of rows of an identity matrices of appropriate dimension, indicating which line power flow measurements are actually taken. Together, $P D B^T \theta$ is a vector of the power flows on the transmission lines to be measured. Analogously, the matrix $Q$ selects the bus power injection measurements that are taken. $Q B D B^T \theta$ is a vector of power injections at the buses to be measured. Therefore, $H$ is the measurement matrix, relating the measured power quantities to the network states. The number of rows of $H$ is denoted $m$.

The measurements $z$ and the network information $H$ are jointly used to find an estimate of the network states denoted as $\hat{\theta}$. Assuming that the network is observable, it is well-established that the state estimate can be obtained using the weighted least squares approach \cite[Chapter 5]{Abur_Exposito_SEbook} \cite[Chapter 8]{Monticelli_SEbook}:
\begin{equation} \label{def:theta_hat}
  \hat{\theta} = {(H^T W H)}^{-1} W H^T z,
\end{equation}
where $W$ is a positive definite diagonal weighting matrix, typically weighting more on the more accurate measurements. The state estimate $\hat{\theta}$ is subsequently fed to other vital SCADA functionalities such as OPF dispatch and CA. Therefore, the accuracy and reliability of $\hat{\theta}$ is of paramount concern.

To detect possible faults in the measurements $z$, the BDD test is commonly performed (see \cite{Abur_Exposito_SEbook,Monticelli_SEbook}). In one typical strategy, if the norm of the residual
\begin{equation} \label{def:BDD_residual_norm}
  {\rm residual} \triangleq z - H \hat{\theta} = (I - H{(H^T W H)}^{-1} W H^T) \Delta z\;
\end{equation}
is too big, then the BDD alarm will be triggered.

\subsection{Unobservable data attack and security index}
The BDD test is in general sufficient to detect the presence of $\Delta z$ if it contains a single random error \cite{Abur_Exposito_SEbook,Monticelli_SEbook}. However, in face of a \emph{coordinated} malicious data attack on multiple measurements the BDD test can fail. In particular, \cite{LRN09} considers unobservable attack of the form
\begin{equation} \label{def:unobservable_attack}
\Delta z = H \Delta \theta
\end{equation}
for an arbitrary $\Delta \theta \in \mathbb{R}^n$. Since $\Delta z$ as defined in (\ref{def:unobservable_attack}) would result in a zero residual in (\ref{def:BDD_residual_norm}), it is unobservable from the BDD perspective. This was also experimentally verified in \cite{TDSJ11} in a realistic SCADA system testbed. To quantify the vulnerability of a network to unobservable attacks, \cite{STJ_SCS2010} introduced the notion of security index for an arbitrarily specified measurement. The security index is the optimal objective value of the following cardinality minimization problem:
\begin{equation} \label{opt:card_min_con}
  \begin{array}{cl}
    \mathop{\rm minimize}\limits_{\Delta \theta \in \mathbb{R}^n} & {\left\| H \Delta \theta \right\|}_0 \vspace{1mm} \\
    \textrm{subject to} & H(k,:) \Delta \theta = 1,
  \end{array}
\end{equation}
where $k$ is given, indicating that the security index is computed for measurement $k$. The symbol ${\|\cdot\|}_0$ denotes the cardinality of a vector and $H(k,:)$ denotes the $k^{\rm th}$ row of $H$. The security index is the minimum number of measurements an attacker needs to compromise in order to attack measurement $k$ undetected. In particular, a small security index for a particular measurement $k$ means that in order to compromise $k$ undetected it is necessary to compromise only a small number of additional measurements. This can imply that measurement $k$ is relatively easy to compromise in an unobservable attack. As a result, the knowledge of the security indices allows the network operator to pinpoint the security vulnerabilities of the network, and to better protect the network with limited resource. To model the case where certain measurements are protected (hence cannot be attacked) \cite{LRN09,DS_SGC2010,BRWKNO10}, problem (\ref{opt:card_min_con}) becomes:
\begin{equation} \label{opt:card_min_con_protected}
  \begin{array}{cl}
    \mathop{\rm minimize}\limits_{\Delta \theta \in \mathbb{R}^n} & {\left\| H \Delta \theta \right\|}_0 \vspace{1mm} \\
    \textrm{subject to} & H(k,:) \Delta \theta = 1 \vspace{1mm} \\
    & H(\mathcal{I},:) \Delta \theta = 0,
  \end{array}
\end{equation}
where the protection index set $\mathcal{I} \subset \{1,2,\ldots,m\}$ is given, $H(\mathcal{I},:)$ denotes a submatrix of $H$ with rows indexed by $\mathcal{I}$. By convention, the constraint $H(\mathcal{I},:) \Delta \theta = 0$ is ignored when $\mathcal{I} = \emptyset$. Hence, (\ref{opt:card_min_con}) is a special case of (\ref{opt:card_min_con_protected}).

\subsection{Measurement set robustness analysis} \label{subsec:msra}
Problem (\ref{opt:card_min_con}) is also motivated from another important state estimation analysis problem, namely observability analysis \cite{Abur_Exposito_SEbook,Monticelli_SEbook}. The measurement set, described by $H$ in (\ref{def:H_matrix}), is observable if $\hat{\theta}$ can be uniquely determined by (\ref{def:theta_hat}). An important question of observability analysis is as follows:
\begin{equation} \label{opt:card_min_obsv}
  \begin{array}{cl}
    \mathop{\rm minimize}\limits_{\mathcal{J}} & {\left\|\mathcal{J}\right\|}_0 \vspace{0.5mm} \\
    \textrm{subject to} & {\rm rank}(H(\bar{\mathcal{J}},:)) < n \vspace{1.5mm} \\
    & {\rm rank}(H(\bar{\mathcal{J}}\cup\{k\},:)) = n \vspace{1.5mm} \\
    & k \in \mathcal{J}.
  \end{array}
\end{equation}
In above, $k$ is a given index and $\bar{\mathcal{J}}$ denotes the complement of $\mathcal{J}$ (for index set $\mathcal{I}$ in the rest of the paper, $\bar{\mathcal{I}}$ denotes its complement). The meaning of (\ref{opt:card_min_obsv}) is as follows: $\mathcal{J}$ denotes a subset of measurements from the measurement system described by $H$. The condition that ${\rm rank}(H(\bar{\mathcal{J}},:)) < n$ means that the measurement system becomes unobservable if the measurements associated with $\mathcal{J}$ are lost. That is, it becomes impossible to uniquely determine $\hat{\theta}$ from $H(\bar{\mathcal{J}},:)$. The problem in (\ref{opt:card_min_obsv}) seeks the minimum cardinality $\mathcal{J}$ which must include a particular given measurement $k$. Therefore, if there exist a measurement $k$ which leads to an instance of (\ref{opt:card_min_obsv}) with a very small objective value, then the measurement system is not robust against meter failure. Special cases of (\ref{opt:card_min_obsv}) have been extensively studied in the power system community. For instance, the solution label sets of cardinalities one and two are, respectively, referred to as critical measurements and critical sets containing measurement $k$ \cite{Abur_Exposito_SEbook}. Their calculations have been documented in, for example, \cite{Abur_Exposito_SEbook,KC91,AAG09,AH86,CKD81}. For the more general cases where the minimum cardinality is $p > 2$, the solution label set in (\ref{opt:card_min_obsv}) is a critical $p$-tuple which contains the specified measurement $k$ \cite{LAB00,SSJ_ckt}. Solving (\ref{opt:card_min_con}) solves (\ref{opt:card_min_obsv}) as well. The justification is given by the following statement inspired by \cite{KJTT11}, and proved in Appendix:
\begin{proposition} \label{thm:con_obsv_eq}
  Let $H \in \mathbb{R}^{m \times n}$ and $k \in \{1,2,\ldots,m\}$ be given for problems (\ref{opt:card_min_con}) and (\ref{opt:card_min_obsv}). Denote the two conditions:
  \begin{description}
    \item[I:] $H(k,:) \neq 0$.
    \item[II:] $H$ has full column rank ($=n$).
  \end{description}
  The following three statements are true:
  \begin{description}
    \item[(a)] Problem (\ref{opt:card_min_con}) is feasible if and only if condition I is satisfied.
    \item[(b)] Problem (\ref{opt:card_min_obsv}) is feasible if and only if conditions I and II are satisfied.
    \item[(c)] If conditions I and II are satisfied, then (\ref{opt:card_min_con}) and (\ref{opt:card_min_obsv}) are equivalent (see Definition \ref{def:equivalence} in Section~\ref{subsec:definition}).
  \end{description}
\end{proposition}
Note that if condition I is not satisfied, then the corresponding measurement $k$ should be removed from consideration. Also, since measurement redundancy is a common practice in power networks \cite{Abur_Exposito_SEbook,Monticelli_SEbook}, $H$ can be assumed to have full column rank ($=n$). Therefore, conditions I and II in Proposition~\ref{thm:con_obsv_eq} can be justified in practice. Finally, note that Proposition~\ref{thm:con_obsv_eq} remains true for arbitrary matrix $H$ (not necessarily defined by (\ref{def:H_matrix})).

\section{Problem Statement and Main Result} \label{sec:main_contribution}
\subsection{Problem statement}
As discussed previously, this paper proposes an efficient solution to the security index (i.e., attack construction) problem in (\ref{opt:card_min_con_protected}). However, the proposed result focuses only on a generalization of a special case of (\ref{opt:card_min_con_protected}). In this special case, $H$ in (\ref{def:H_matrix}) does not contain injection measurements:
\begin{equation} \label{eqn:no_inj_assumption}
  H = P D B^T.
\end{equation}
The limitation of the assumption in (\ref{eqn:no_inj_assumption}) will be discussed in Section~\ref{subsec:H_assumptions}, after the main result is presented.

In the Appendix, it is shown that the special case of (\ref{opt:card_min_con_protected}) with the assumption in (\ref{eqn:no_inj_assumption}) is equivalent to
\begin{equation} \label{opt:card_min_con_var}
  \begin{array}{cl}
    \mathop{\rm minimize}\limits_{\Delta \theta \in \mathbb{R}^n} & \left\| P(\bar{\mathcal{I}},:) B^T \Delta \theta \right\|_0 \vspace{1mm} \\
    \textrm{subject to} & P(k,:) B^T \Delta \theta = 1 \vspace{1mm} \\
    & P(\mathcal{I},:) B^T \Delta \theta = 0.
  \end{array}
\end{equation}
Instead of considering (\ref{opt:card_min_con_var}) directly, the proposed result pertains to a more general optimization problem associated with a \emph{totally unimodular matrix} (i.e., the determinant of every square submatrix is either $-1$, 0, or 1 \cite{ACCO_Schrijver10}). In particular, the following problem is the main focus of this paper:
\begin{equation}\label{opt:l0}
  \begin{array}{cl}
    \mathop{\rm minimize}\limits_{x \in \mathbb{R}^n} & \left\|A(\bar{\mathcal{I}},:) x\right\|_0 \\
    \textrm{subject to} & A(k,:) x = 1 \vspace{1mm} \\
    & A(\mathcal{I},:) x = 0,
  \end{array}
\end{equation}
where $A \in \mathbb{R}^{m \times n}$ is a given totally unimodular matrix, and $k \in \{1,2,\ldots,m\}$ and $\mathcal{I} \subset \{1,2,\ldots,m\}$ are given. Since $B$ in (\ref{opt:card_min_con_var}) is an incidence matrix, $P B^T$ is a totally unimodular matrix. Therefore, (\ref{opt:l0}) is a generalization of (\ref{opt:card_min_con_var}). However, neither (\ref{opt:l0}) nor (\ref{opt:card_min_con_protected}) includes each other as special cases.

\subsection{$l_1$ relaxation}
Problem (\ref{opt:l0}) is a cardinality minimization problem. In general, no efficient algorithms have been found for solving cardinality minimization problems \cite{CT05}, so heuristic or relaxation based algorithms are often considered. The $l_1$ relaxation (i.e., basis pursuit \cite{CDS_Basis_Pursuit}) is a relaxation technique which has received much attention. In $l_1$ relaxation, instead of (\ref{opt:l0}), the following optimization problem is set up and solved:
\begin{equation}\label{opt:l1}
  \begin{array}{cl}
   \mathop{\rm minimize}\limits_{x \in \mathbb{R}^n} & \left\|A(\bar{\mathcal{I}},:) x\right\|_1 \\
    \textrm{subject to} & A(k,:) x = 1 \vspace{1mm} \\
    & A(\mathcal{I},:) x = 0,
  \end{array}
\end{equation}
where in the objective function in (\ref{opt:l1}) the vector 1-norm replaces the cardinality in (\ref{opt:l0}). Problem (\ref{opt:l1}) can be rewritten as a linear programming (LP) problem in standard form \cite[pp.~4-6, p.17]{BT97}:
\begin{equation} \label{opt:l1_LP}
    \begin{array}{cl}
      \mathop{\rm minimize}\limits_{x_+,x_-,y_+,y_-} & \sum\limits_{j = 1}^{|\bar{\mathcal{I}}|} \big(y_+(j) + y_-(j)\big) \\
      \textrm{subject to} & A(\bar{\mathcal{I}},:)(x_+ - x_-) = y_+ - y_- \\
      & A(k,:)(x_+ - x_-) = 1 \\
      & A(\mathcal{I},:)(x_+ - x_-) = 0 \\
      & x_+ \in \mathbb{R}^n_+, \; x_- \in \mathbb{R}^n_+, \; y_+ \in \mathbb{R}^{|\bar{\mathcal{I}}|}_+, \; y_- \in \mathbb{R}^{|\bar{\mathcal{I}}|}_+,
    \end{array}
\end{equation}
where $|\bar{\mathcal{I}}|$ denotes the cardinality of the index set $\bar{\mathcal{I}}$.

If $(x_+,x_-,y_+,y_-)$ is a feasible solution to (\ref{opt:l1_LP}), then $x \triangleq x_+ - x_-$ is feasible to (\ref{opt:l0}). Hence, an optimal solution to (\ref{opt:l1_LP}), if it exists, corresponds to a suboptimal solution to the original problem in (\ref{opt:l0}). An important question is under what conditions this suboptimal solution is actually optimal to (\ref{opt:l0}). An answer is provided by our main result, based on the special structure in (\ref{opt:l0}) and the fact that matrix $A$ is totally unimodular.
\subsection{Statement of main result}
\begin{theorem} \label{thm:l0_l1_solution}
Let $(x_+^\star,x_-^\star,y_+^\star,y_-^\star)$ be an \emph{optimal basic feasible solution} to (\ref{opt:l1_LP}), where $A$, $k$ and $\mathcal{I}$ are defined in (\ref{opt:l0}). Then $x^\star \triangleq (x_+^\star - x_-^\star)$ is an optimal solution to (\ref{opt:l0}).
\end{theorem}

\begin{remark} \label{rmk:LP_procedure}
  Theorem~\ref{thm:l0_l1_solution} provides a complete procedure for solving (\ref{opt:l0}) via (\ref{opt:l1_LP}). If the standard form LP problem in (\ref{opt:l1_LP}) is feasible, then it contains at least one basic feasible solution (see the definition in Section~\ref{subsec:definition}). Together with the fact that the objective value is bounded from below (e.g., by zero), \cite[Theorem~2.8, p.~66]{BT97} implies that problem~(\ref{opt:l1_LP}) contains at least one optimal basic feasible solution, which can be used to construct an optimal solution to (\ref{opt:l0}) according to Theorem~\ref{thm:l0_l1_solution}. Conversely, if the feasible set of (\ref{opt:l1_LP}) is empty, then the feasible set of (\ref{opt:l0}) must also be empty because a feasible solution to (\ref{opt:l0}) can be used to construct a feasible solution to (\ref{opt:l1_LP}).
\end{remark}
\begin{remark} \label{rmk:LP_simplex}
  To ensure that an optimal basic feasible solution to (\ref{opt:l1_LP}) is found if one exists, the simplex method (e.g., \cite[Chapter~3]{BT97}) can be used to solve (\ref{opt:l1_LP}).
\end{remark}

The proof of Theorem~\ref{thm:l0_l1_solution} will be given in Section~\ref{sec:proof}. Before that, the related work are reviewed, and the assumption in (\ref{eqn:no_inj_assumption}) is discussed.

\section{Related Work} \label{sec:related_work}

\subsection{Rationale of the no injection assumption in (\ref{eqn:no_inj_assumption})} \label{subsec:H_assumptions}
Consider the case of (\ref{opt:card_min_con_protected}) where $\mathcal{I}$ corresponds only to line power flow measurements, then with the definition of $H$ in (\ref{def:H_matrix}) it can be verified that (\ref{opt:card_min_con_protected}) is equivalent to the following:
\begin{equation} \label{opt:card_min_con_obj_sum}
  \begin{array}{cl}
    \mathop{\rm minimize}\limits_{\Delta \theta \in \mathbb{R}^n} & \left\| P(\bar{\mathcal{I}},:) B^T \Delta \theta \right\|_0 + \left\| Q B^T D B \Delta \theta \right\|_0 \vspace{1mm} \\
    \textrm{subject to} & P(k,:) B^T \Delta \theta = 1 \vspace{1mm} \\
    & P(\mathcal{I},:) B^T \Delta \theta = 0.
  \end{array}
\end{equation}
This indicates that the considered problem in (\ref{opt:card_min_con_var}) is a relaxation \cite{BT97} of the general case in (\ref{opt:card_min_con_obj_sum}). \cite{SSJ_CDC2011_mincut} utilizes this observation and obtains satisfactory suboptimal solution to (\ref{opt:card_min_con_protected}). Alternatively, \cite{SSJ_ckt} considers indirectly accounting for the term $\left\| Q B^T D B \Delta \theta \right\|_0$ in the objective function of (\ref{opt:card_min_con_obj_sum}). \cite{SSJ_ckt} demonstrates that solving the following problem provides satisfactory suboptimal solution to (\ref{opt:card_min_con_obj_sum})
\begin{equation} \label{opt:card_min_con_var1}
  \begin{array}{cl}
    \mathop{\rm minimize}\limits_{\Delta \theta \in \mathbb{R}^n} & \left\| \tilde{P}(\bar{\mathcal{\tilde{I}}},:) B^T \Delta \theta \right\|_0 \vspace{1mm} \\
    \textrm{subject to} & \tilde{P}(\tilde{k},:) B^T \Delta \theta = 1 \vspace{1mm} \\
    & \tilde{P}(\mathcal{\tilde{I}},:) B^T \Delta \theta = 0,
  \end{array}
\end{equation}
with appropriately defined $\tilde{P}$, $\tilde{I}$ and $\tilde{k}$. Notice that (\ref{opt:card_min_con_var1}) has the same form as (\ref{opt:card_min_con_var}). In conclusion, the ``no injection assumption'' in (\ref{eqn:no_inj_assumption}) which leads to (\ref{opt:card_min_con_var}) introduces limitation, but it need not be as restrictive as it might appear. The proposed result in Theorem~\ref{thm:l0_l1_solution} still leads to a LP based approach to obtain suboptimal solutions to (\ref{opt:card_min_con_obj_sum}) (and hence (\ref{opt:card_min_con_protected})). 

\subsection{Relationship with minimum cut based results} \label{subsec:relationship_mincut}
Nevertheless, the main strength of the current result lies in the fact that it solves problem (\ref{opt:l0}) where the $A$ matrix is totally unimodular. (\ref{opt:l0}) includes (\ref{opt:card_min_con_var}) as a special case where the corresponding constraint matrix is a transposed graph incidence matrix. This distinguishes the current work with other ones such as \cite{SSJ_CDC2011_mincut,SSJ_ckt,HJJSS12} which specialize in solving (\ref{opt:card_min_con_var}) using graph-based minimum cut algorithms (e.g., \cite{Stoer:1997:SMA:263867.263872}). One example of $A$ which is totally unimodular but not associated with a graph is the matrix with consecutive ones property (i.e., if either for each row or for each column, the 1's appear consecutively) \cite{Schrijver98}. For a possible application, consider a networked control system \cite{HNY07,NCS_BHJ} with one controller and $n$ sensor nodes. Each node contains a scalar state value, constant over a period of $m$ time slots. The nodes need to transmit their state values through a shared channel to the controller. Each node can keep transmitting over an arbitrary period of consecutive time slots. At each time slot, the measurement transmitted to the controller is the sum of the state values of all transmitting nodes. Denote $z \in \mathbb{R}^m$ as the vector of measurements transmitted over all time slots, and $\theta \in \mathbb{R}^n$ as the vector of node state values. Then the measurements and the states are related by $z = A \theta$, where $A \in \mathbb{R}^{m \times n}$ is a $(0,1)$ matrix with consecutive ones in the each column. Solving the observability problem in (\ref{opt:card_min_obsv}) with $H = A$ can identify the vulnerable measurement slots, which should have higher priority in communication.

\subsection{Relationship with compressed sensing type results}
Problem (\ref{opt:l0}) can be written in a form more common in the literature. Consider only the case where the null space of $A^T$ is not empty (otherwise ${\rm rank}(A) = m$ and (\ref{opt:l0}) is trivial). With a change of decision variable $z = A x$, (\ref{opt:l0}) can be posed as:
\begin{equation} \label{opt:l0_reformulated}
  \begin{array}{cl}
    \mathop{\rm minimize}\limits_{z \in \mathbb{R}^m} & {\|z(\bar{\mathcal{I}})\|}_0 \\
    \textrm{subject to} & \,\,\,\, L z = 0 \\
    & z(\mathcal{I}) = 0 \\
    & z(k) = 1,
  \end{array}
\end{equation}
where $L$ has full rank and $L A = 0$, and $z(\bar{\mathcal{I}})$ denotes a sub-vector of $z$ containing the entries corresponding to the index set $\bar{\mathcal{I}}$. (\ref{opt:l0_reformulated}) can be written as the cardinality minimization problem considered, for instance, in \cite{CT05,DE03,CWS_Weight_l1,BDE09_siam_review}:
\begin{equation} \label{opt:min_card}
  \begin{array}{cl}
    \mathop{\rm minimize}\limits_{\tilde{z}} & {\|\tilde{z}\|}_0 \\
    \textrm{subject to} & \Phi \tilde{z} = b,
  \end{array}
\end{equation}
with appropriately defined matrix $\Phi$ and vector $b$. In this subsection, we restrict the discussion to the standard case. That is, (\ref{opt:min_card}) is feasible and $\Phi$ is a full rank matrix with more columns than rows. As (\ref{opt:min_card}) is well-studied, certain conditions regarding when its optimal solution can be obtained by $l_1$ relaxation are known. For example, \cite{DE03,GN03} report a sufficient condition based on mutual coherence, which is denoted as $\mu(\Phi)$ and defined as
\begin{equation} \label{def:mutual_coherence}
  \mu(\Phi) = \max\limits_{i \neq j} \frac{|{\Phi(:,i)}^T \Phi(:,j)|}{{\|\Phi(:,i)\|}_2 {\|\Phi(:,j)\|}_2}. \vspace{1mm}
\end{equation}
The sufficient condition \cite{BDE09_siam_review} states that if there exists a feasible solution $\tilde{z}$ in (\ref{opt:min_card}) which is sparse enough:
\begin{equation} \label{eqn:sb_mc}
  {\left\|\tilde{z}\right\|}_0 < \frac{1}{2} \Big(1+\frac{1}{\mu(\Phi)} \Big),
\end{equation}
then $\tilde{z}$ is the unique optimal solution to (\ref{opt:min_card}) and its $l_1$ relaxation (i.e., problem (\ref{opt:min_card}) with ${\|\tilde{z}\|}_1$ replacing ${\|\tilde{z}\|}_0$). Another well-known sufficient condition is based on the restricted isometry property (RIP) \cite{CT05,Candes2008589}. For any integer $s$, the RIP constant $\delta_s$ of matrix $\Phi$ is the smallest number satisfying
\begin{equation} \label{def:RIP}
  (1-\delta_s) {\|x\|}_2^2 \le {\left\|\Phi x\right\|}_2^2 \le (1+\delta_s) {\|x\|}_2^2
\end{equation}
for all vector $x$ such that ${\|x\|}_0 \le s$. The RIP-based sufficient condition \cite{Candes2008589} states that if for some $s$, $\Phi$ has a RIP constant $\delta_{2s} < \sqrt{2}-1$, then any $\tilde{z}$ satisfying $\Phi \tilde{z} = b$ and ${\|\tilde{z}\|}_0 \le s$ is necessarily the unique optimal solution to both (\ref{opt:min_card}) and its $l_1$ relaxation. It has been shown that certain type of randomly generated matrices satisfy the above conditions with overwhelming probabilities (e.g., \cite{CT05} provides a RIP-related result). However, the above conditions might not apply to (\ref{opt:l0}), which is the focus of this paper. For instance, consider $A$ in (\ref{opt:l0}) being a submatrix of the transpose of the incidence matrix of the 6-bus power network from \cite{WW96}:
\begin{displaymath}
  A = \begin{bmatrix}
    1 & -1 & 0 & 0 & 0 & 0 \\
    0 & 1 & -1 & 0 & 0 & 0 \\
    0 & 0 & 0 & 1 & -1 & 0 \\
    0 & 0 & 0 & 0 & 1 & -1 \\
    1 & 0 & 0 & -1 & 0 & 0 \\
    0 & 1 & 0 & 0 & -1 & 0 \\
    0 & 0 & 1 & 0 & 0 & -1
  \end{bmatrix}.
\end{displaymath}
Let $k = 6$, and $\mathcal{I} = \emptyset$. Then the corresponding $\Phi$ in (\ref{opt:min_card}) and $b$ are
\begin{equation} \label{eqn:Phi_b_counterex}
  \Phi = \begin{bmatrix}
    1 & 1 & -1 & -1 & -1 & 0 & 1 \\
    1 & 0 & -1 & 0 & -1 & 1 & 0 \\
    0 & 0 & 0 & 0 & 0 & 1 & 0
  \end{bmatrix}\quad
  b = \begin{bmatrix}
    0 \\ 0 \\ 1
  \end{bmatrix}.
\end{equation}
For this $\Phi$, (\ref{def:mutual_coherence}) implies that $\mu(\Phi) = 1$. Therefore, the sparsity bound in (\ref{eqn:sb_mc}) becomes ${\|\tilde{z}\|}_0 < 1$. This is too restrictive to be practical. Similarly, for all $s \ge 1$, the RIP constants $\delta_{2s}$ are at least one because $\Phi(:,1) = -\Phi(:,3)$. Hence the RIP-based sufficient condition would not be applicable either. Nevertheless, the failure to apply these sufficient conditions here does not mean that it is impossible to show that $l_1$ relaxation can exactly solve (\ref{opt:min_card}). The mutual coherence and RIP-based conditions characterize when a \emph{unique} optimal solution exists for both (\ref{opt:min_card}) and its $l_1$ relaxation, while in this paper uniqueness is not required. Indeed, for (\ref{opt:min_card}) with $\Phi$ and $b$ defined in (\ref{eqn:Phi_b_counterex}), both ${\begin{bmatrix} -1 & 0 & 0 & -1 & 0 & 1 & 0 \end{bmatrix}}^T$ and ${\begin{bmatrix} -1 & 1 & 0 & 0 & 0 & 1 & 0 \end{bmatrix}}^T$ are optimal (this can be verified by inspection). Using the CPLEX LP solver \cite{CPLEX} in MATLAB to solve the $l_1$ relaxation leads to the first optimal solution. It is the main contribution of this paper to show that this is the case in general when (\ref{opt:min_card}) is defined by (\ref{opt:l0}), even though the optimal solution might not be unique. The reason why the proposed result is applicable is that it is based on a polyhedral combinatorics argument, which is different from those of the mutual coherence and RIP based results.


\section{Proof of the Main Result} \label{sec:proof}

\subsection{Definitions} \label{subsec:definition}
The proof requires the following definitions:
\begin{definition} \label{def:equivalence}
  Two optimization problems are {\bf equivalent} if there is an one-to-one correspondence of their instances. The corresponding instances either are both infeasible, both unbounded or both have optimal solutions. In the last case, it is possible to construct an optimal solution to one problem from an optimal solution to the other problem and vice versa. In addition, the two problems have the same optimal objective value.
\end{definition}


\begin{definition}
  A {\bf polyhedron} in $\mathbb{R}^p$ is a subset of $\mathbb{R}^p$ described by linear equality and inequality constraints. A standard form polyhedron (as associated with a standard form LP problem instance) is specified by $\left\{\theta \; \vline \; C \theta = d, \; \theta \ge 0\right\}$ for some given matrix $C$ and vector $d$.
\end{definition}

\begin{definition}
  A {\bf basic solution} \cite[p.~50]{BT97} of a polyhedron in $\mathbb{R}^p$ is a vector satisfying all equality constraints. In addition, out of all active constraints $p$ of them are linearly independent. For a standard form polyhedron with a constraint matrix of full row rank, basic solutions can alternatively be defined by the following statement \cite[p.~53]{BT97}:
\begin{theorem} \label{thm:BFS_standard_form}
   Consider a polyhedron $\left\{\theta \; \vline \; C \theta = d, \; \theta \ge 0\right\}$, and assume that $C \in \mathbb{R}^{l \times p}$ and $C$ has full row rank. A vector $\theta$ is a basic solution if and only if $C \theta = d$ and there exists an index set $\mathcal{J} \subset \{1,2,\ldots,p\}$, with $|\mathcal{J}| = l$, such that $\det(C(:,\mathcal{J})) \neq 0$ and $\theta(i) = 0$ if $i \notin \mathcal{J}$.
\end{theorem}
\end{definition}

\begin{definition} \label{def:BFS}
  A {\bf basic feasible solution} \cite[p.~50]{BT97} of a polyhedron is a basic solution which is also feasible. By convention, the terminology ``a basic feasible solution to a LP problem instance'' should be understood as a basic feasible solution of the polyhedron which defines the feasible set of the instance.
\end{definition}

\subsection{Proof} \label{subsec:proof}
Two lemmas, key to the proof, are presented first. The first lemma states that problem (\ref{opt:l1_LP}), as set up by $l_1$ relaxation, has integer-valued optimal basic feasible solutions.
\begin{lemma} \label{thm:L1_int_sol}
Let $(x_+^\star,x_-^\star,y_+^\star,y_-^\star)$ be an optimal basic feasible solution to (\ref{opt:l1_LP}). Then it holds that $x^\star(i) \triangleq (x_+^\star(i) - x_-^\star(i)) \in \{-1,0,1\}$ for all $1 \le i \le n$. In addition, $y^\star(j) \triangleq |A(\bar{\mathcal{I}}(j),:) x^\star| \in \{0,1\}$ for all $1 \le j \le |\bar{\mathcal{I}}|$, where $\bar{\mathcal{I}}(j)$ denotes the $j^{\rm th}$ element of $\bar{\mathcal{I}}$.
\end{lemma}
\begin{proof}
   Assume that the feasible set of (\ref{opt:l1_LP}) is nonempty, otherwise there is no basic feasible solution (cf.\ Definition~\ref{def:BFS}). The following two claims are made:
   \begin{description}
     \item[(a)] $A(k,:)$ cannot be a linear combination of the rows of $A(\mathcal{I},:)$.
     \item[(b)] There exists ${\mathcal{I}}^\prime \subset \mathcal{I}$ such that either ${\mathcal{I}}^\prime = \emptyset$ or the rows of $A({\mathcal{I}}^\prime,:)$ are linearly independent. In addition, in both cases $A({\mathcal{I}}^\prime,:) \theta = 0$ and $A(\mathcal{I},:) \theta = 0$ define the same constraints.
   \end{description}
   Claims (a) and (b) together imply that problem (\ref{opt:l1_LP}) can be written as a standard form LP problem with a constraint matrix with full row rank (i.e., matrix $C$ below):
   \begin{equation} \label{opt:L1_1_standard_form}
    \begin{array}{cl}
      \mathop{\rm minimize}\limits_{\theta} & f^T \theta \\
      \textrm{subject to} & C \theta = d \vspace{1mm} \\
      & \,\,\,\,\, \theta \ge 0,
    \end{array}
  \end{equation}
   with
\begin{equation} \label{eqn:Cd_data}
\begin{array}{l}
  C \triangleq \begin{bmatrix}
        A(\bar{\mathcal{I}},:) & -A(\bar{\mathcal{I}},:) & -I_{|\bar{\mathcal{I}}|} & I_{|\bar{\mathcal{I}}|} \\
        A({\mathcal{I}}^\prime,:) & -A({\mathcal{I}}^\prime,:) & 0 & 0 \\
        A(k,:) \:  & -A(k,:) \:  & 0 & 0
      \end{bmatrix} \\
       d \triangleq \begin{bmatrix} 0 \\ 0 \\ 1 \end{bmatrix} \; \theta \triangleq \begin{bmatrix}
        x_+ \\ x_- \\ y_+ \\ y_-
      \end{bmatrix} \; f \triangleq \begin{bmatrix}
        0_{n \times 1} \\ 0_{n \times 1} \\ \mathbf{1}_{|\bar{\mathcal{I}}| \times 1} \\ \mathbf{1}_{|\bar{\mathcal{I}}| \times 1}
      \end{bmatrix},
\end{array}
\end{equation}
where $I_{|\bar{\mathcal{I}}|}$ is an identity matrix of dimension $|\bar{\mathcal{I}}|$, and $\mathbf{1}$ is a vector of all ones.

To see the claims, first note that (a) is implied by the feasibility of (\ref{opt:l1_LP}). For (b), If $\mathcal{I} = \emptyset$ or $A(\mathcal{I},:) = 0$, then set ${\mathcal{I}}^\prime = \emptyset$. Otherwise, there exists ${\mathcal{I}}^\prime \subset \mathcal{I}$ with the properties that $|{\mathcal{I}}^\prime| = {\rm rank}(A(\mathcal{I},:))$, $A({\mathcal{I}}^\prime,:)$ has linearly independent rows and $A(\mathcal{I},:) = S A({\mathcal{I}}^\prime,:)$ for some matrix $S$. On the other hand, $A({\mathcal{I}}^\prime,:) = S^\prime A(\mathcal{I},:)$ for some matrix $S^\prime$, because ${\mathcal{I}}^\prime \subset \mathcal{I}$. Hence, $A(\mathcal{I},:) \theta = 0$ and $A({\mathcal{I}}^\prime,:) \theta = 0$ define the same constraints. This shows (b).

The next step of the proof is to show that every basic solution of (\ref{opt:L1_1_standard_form}) has its entries being either $-1$, 0 or 1. Denote the matrix $B_1$ as the first $2n$ columns of $C$, and let $\tilde{B}_1$ be any square submatrix of $B_1$. If $\tilde{B}_1$ has two columns (or rows) which are the same or negative of each other, then $\det(\tilde{B}_1) = 0$. Otherwise, $\tilde{B}_1$ is a (possibly row and/or column permuted) square submatrix of $A$, and $A$ is assumed to be totally unimodular. Hence, $\det(\tilde{B}_1) \in \{-1,0,1\}$ and $B_1$ is totally unimodular. Next consider the matrix $B$ defined as
\begin{displaymath}
\begin{array}{l}
  B \triangleq \begin{bmatrix} C & d \end{bmatrix} =
    \begin{bmatrix}
        A(\bar{\mathcal{I}},:) & -A(\bar{\mathcal{I}},:) & -I_{|\bar{\mathcal{I}}|} & I_{|\bar{\mathcal{I}}|} & 0 \\
        A({\mathcal{I}}^\prime,:) & -A({\mathcal{I}}^\prime,:) & 0 & 0 & 0 \\
        A(k,:) \:  & -A(k,:) \:  & 0 & 0 & 1
      \end{bmatrix} \vspace{1mm} \\
      \;\;\;\; = \begin{bmatrix}
         B_1 & \begin{matrix} -I_{|\bar{\mathcal{I}}|} & I_{|\bar{\mathcal{I}}|} & 0 \\ 0 & 0 & 0 \\ 0 & 0 & 1 \end{matrix}
      \end{bmatrix}.
\end{array}
\end{displaymath}
Denote the number of rows and the number of columns of $B$ as $m_B$ and $n_B$ respectively. Let $\mathcal{J} \subset \{1,2,\ldots,n_B\}$ be any set of column indices of $B$ such that $|\mathcal{J}| = m_B$ (so that $B(:,\mathcal{J})$ is square). If $B(:,\mathcal{J})$ contains only columns of $B_1$, then $\det(B(:,\mathcal{J})) \in \{-1,0,1\}$ since $B_1$ is totally unimodular. Otherwise, by repeatedly applying Laplace expansion on the columns of $B(:,\mathcal{J})$ which are not columns of $B_1$, it can be shown that $\det(B(:,\mathcal{J}))$ is equal to the determinant of a square submatrix of $B_1$, which can only be $-1$, 0 or 1. Hence, by Cramer's rule the following holds: If $v$ is the solution to the following system of linear equations
\begin{equation} \label{eqn:BS}
\begin{array}{l}
    B(:,\mathcal{J}) \; v = B(:,n_B), \; \mathcal{J} \subset \{1,\ldots,n_B-1\}, \\
     \quad \quad \;\; |\mathcal{J}| = m_B, \; {\rm and} \; \det(B(:,\mathcal{J})) \neq 0,
\end{array}
\end{equation}
then
\begin{equation} \label{eqn:BS_int}
    v(j) \in \{-1,0,1\}, \quad \forall \; j.
\end{equation}
Theorem~\ref{thm:BFS_standard_form} and (\ref{eqn:BS}) together imply that the nonzero entries of all basic solutions to (\ref{opt:L1_1_standard_form}) are either $-1$, 0 or 1. Therefore, the basic feasible solutions, which are also basic solutions, to the polyhedron in (\ref{opt:L1_1_standard_form}) also satisfy this integrality property.

Finally, let $(x_+^\star,x_-^\star,y_+^\star,y_-^\star)$ be an optimal basic feasible solution. Then feasibility (i.e., nonnegativity) implies that
\begin{equation} \label{eqn:x_star_L1_1_standard}
\begin{array}{l}
     x_+^{\star}(j) \in \{0,1\}, \quad x_-^{\star}(j) \in \{0,1\}, \\
     y_+^{\star}(j) \in \{0,1\}, \quad y_-^{\star}(j) \in \{0,1\}, \quad \forall \; j.
\end{array}
\end{equation}
The minimization excludes the possibility that, at optimality, $y_+^\star(j) = y_-^\star(j) = 1$. Hence, it is possible to define $x^\star$ and $y^\star$ such that
\begin{equation} \label{eqn:x_star_L1_1}
\begin{array}{l}
  x^\star(i) \triangleq (x_+^{\star}(i) - x_-^{\star}(i)) \in \{-1,0,1\} \quad \forall \; i \\
  y^\star(j) \triangleq (y_+^\star(j) + y_-^\star(j)) = |A(\bar{\mathcal{I}}(j),:) \; x^{\star}| \in \{0,1\} \quad \forall \; j.
\end{array}
\end{equation}
\end{proof}

The second lemma is concerned with a restricted version of (\ref{opt:l0}) with an infinity norm bound as follows:
\begin{equation} \label{opt:l0_NB}
  \begin{array}{cl}
      \mathop{\rm minimize}\limits_x & \left\|A(\bar{\mathcal{I}},:) \: x \right\|_0 \\
      \textrm{subject to} & A(k,:) \:  x = 1 \\
      & A(\mathcal{I},:) \: x = 0 \\
      & \,\,\,\, \left\|A x \right\|_\infty \le 1.
  \end{array}
\end{equation}
\begin{lemma} \label{thm:linf_bound}
  Optimization problems (\ref{opt:l0}) and (\ref{opt:l0_NB}) are equivalent.
\end{lemma}
\begin{proof}
  Suppose (\ref{opt:l0}) is feasible, then it has an optimal solution denoted as $x^\star$. Let $\bar{\mathcal{I}}_{x^\star} \subset \bar{\mathcal{I}}$ be the row index set such that $A(j,:) x^\star \neq 0$ if and only if $j \in \bar{\mathcal{I}}_{x^\star}$. Then it is claimed that there exists a common optimal solution to both (\ref{opt:l0}) and (\ref{opt:l0_NB}) with the same optimal objective value. The argument is as follows. The property of $x^\star$ implies the feasibility of ${(\ref{opt:l1_LP})}^\prime$, which is denoted as a variant of (\ref{opt:l1_LP}) with $\bar{\mathcal{I}}$ replaced by $\bar{\mathcal{I}}_{x^\star}$. By \cite[Corollary~2.2, p.~65]{BT97}, problem~${(\ref{opt:l1_LP})}^\prime$, as a standard form LP problem, has at least one basic feasible solution. Furthermore, since the optimal objective value of ${(\ref{opt:l1_LP})}^\prime$ is bounded from below (e.g., by zero), \cite[Theorem~2.8, p.~66]{BT97} implies that ${(\ref{opt:l1_LP})}^\prime$ has an optimal basic feasible solution $(\tilde{x}_+^\star,\tilde{x}_-^\star,\tilde{y}_+^\star,\tilde{y}_-^\star)$ which is integer-valued as specified by Lemma~\ref{thm:L1_int_sol}. Denote $\tilde{x}^\star \triangleq (\tilde{x}_+^\star - \tilde{x}_-^\star)$. Then $\tilde{x}^\star$ is feasible to both (\ref{opt:l0}) and (\ref{opt:l0_NB}) since $\bar{\mathcal{I}}_{x^\star} \subset \bar{\mathcal{I}}$, $|A(\bar{\mathcal{I}}_{x^\star},:) \tilde{x}^\star| \in {\{0,1\}}^{|\bar{\mathcal{I}}_{x^\star}|}$ and $k \in \bar{\mathcal{I}}_{x^\star}$. Also, $\left\|A(\bar{\mathcal{I}},:) \tilde{x}^\star \right\|_0 = \left\|A(\bar{\mathcal{I}}_{x^\star},:) \tilde{x}^\star \right\|_0 \le \left\|A(\bar{\mathcal{I}}_{x^\star},:) x^\star \right\|_0 = \left\|A(\bar{\mathcal{I}},:) x^\star\right\|_0$, as the inequality is true because $\tilde{x}^\star$ is an optimal solution to ${(\ref{opt:l1_LP})}^\prime$. Hence $\tilde{x}^\star$ is optimal to both (\ref{opt:l0}) and (\ref{opt:l0_NB}), with the same objective value.

  Conversely, suppose (\ref{opt:l0}) is infeasible, then (\ref{opt:l0_NB}) is also infeasible. This concludes that (\ref{opt:l0}) and (\ref{opt:l0_NB}) are equivalent.
\end{proof}

\begin{proof}[Proof of Theorem~\ref{thm:l0_l1_solution}]
  Let $(x_+^\star,x_-^\star,y_+^\star,y_-^\star)$ be an optimal basic feasible solution to (\ref{opt:l1_LP}). Then there exist $x^\star$ and $y^\star$ as defined in Lemma~\ref{thm:L1_int_sol}. In particular, $x^\star = (x_+^\star - x_-^\star)$. It can be verified that $(x^\star,y^\star)$ is an optimal solution to the following optimization problem:
  \begin{displaymath}
    \begin{array}{cl}
      \mathop{\rm minimize}\limits_{x,y} & \quad \sum\limits_{j = 1}^{|\bar{\mathcal{I}}|} y(i) \\
      \textrm{subject to} & \,\,\,\, A(\bar{\mathcal{I}},:) x \le y \\
      & \! -A(\bar{\mathcal{I}},:) x \le y \\
      & \,\,\,\, A(k,:) x = 1 \\
      & \,\,\,\, A(\mathcal{I},:) x = 0 \\
      & \,\,\, 0 \le y(j) \le 1 \quad \forall \; j = 1,2,\ldots,|\bar{\mathcal{I}}|,
    \end{array}
  \end{displaymath}
  where the inequalities above hold entry-wise. Because of the property that $y^\star(j) \in \{0,1\}$ for all $j$, $(x^\star,y^\star)$ is also an optimal solution to
  \begin{equation} \label{opt:l0_proof}
    \begin{array}{cl}
      \mathop{\rm minimize}\limits_{x,y} & \quad \sum\limits_{j = 1}^{|\bar{\mathcal{I}}|} y(i) \\
      \textrm{subject to} & \,\,\,\, A(\bar{\mathcal{I}},:) x \le y \\
      & \! -A(\bar{\mathcal{I}},:) x \le y \\
      & \,\,\,\, A(k,:) x = 1 \\
      & \,\,\,\, A(\mathcal{I},:) x = 0 \\
      & \,\,\,\, y(j) \in \{0,1\} \quad \forall \; j = 1,2,\ldots,|\bar{\mathcal{I}}|.
    \end{array}
  \end{equation}
  It can be verified that (\ref{opt:l0_proof}) is equivalent to (\ref{opt:l0_NB}). Then Lemma~\ref{thm:linf_bound} states that (\ref{opt:l0_proof}) is also equivalent to (\ref{opt:l0}). Consequently, $(x^\star,y^\star)$ being an optimal solution to (\ref{opt:l0_proof}) implies that (\ref{opt:l0}) is feasible with optimal objective value being $\sum\limits_{j = 1}^{|\bar{\mathcal{I}}|} y^\star(j)$. A feasible solution to (\ref{opt:l0}) is $x^\star$. Since $y^\star(j) = |A(\bar{\mathcal{I}}(j),:) x^\star| \in {\{0,1\}} \; \forall j$, it holds that $\left\|A(\bar{\mathcal{I}},:) x^\star\right\|_0 = \sum\limits_{j = 1}^{|\bar{\mathcal{I}}|} y^\star(j)$. Hence, $x^\star$ is an optimal solution to (\ref{opt:l0}).
\end{proof}

\section{Numerical Demonstration} \label{sec:numerical_demonstration}
As a demonstration, instances of the restricted security index problem in (\ref{opt:card_min_con_var}) are solved with $P$ being an identity matrix and $\mathcal{I}$ being empty. The incidence matrix $B$ describes the topology of one of the following benchmark systems: IEEE 14-bus, IEEE 57-bus, IEEE 118-bus, IEEE 300-bus and Polish 2383-bus and Polish 2736-bus \cite{MATPOWER}. For each benchmark, (\ref{opt:card_min_con_var}) is solved for all possible values of $k$ (e.g., 186 choices in the 118-bus case and 411 choices in the 300-bus case). Two solution approaches are tested. The first approach is the one proposed. It is denoted the $l_1$ approach, and includes the following steps:
\begin{enumerate}
  \item Set up the LP problem in (\ref{opt:l1_LP}) with $A$ being $B^T$.
  \item Solve (\ref{opt:l1_LP}) using a LP solver (e.g., CPLEX LP). Let $(x_+^\star,x_-^\star,y_+^\star,y_-^\star)$ be its optimal solution.
  \item Define $\Delta \theta^\star = x_+^\star - x_-^\star$. It is the optimal solution to (\ref{opt:card_min_con_var}), according to Theorem~\ref{thm:l0_l1_solution}.
\end{enumerate}
The second solution approach to (\ref{opt:card_min_con_var}) is standard, and it was applied also in \cite{SSJ_CDC2011_mincut,SSJ_ckt}. This second approach is referred to as $l_0$ approach, as (\ref{opt:card_min_con_var}) is formulated into the following problem:
\begin{equation} \label{opt:si_MILP}
  \begin{array}{cccl}
    \mathop{\rm minimize}\limits_{\Delta \theta, \; y} & \quad \sum\limits_j y(j) & & \\
    \textrm{subject to} & B^T \Delta \theta & \le & M y \\
    & -B^T \Delta \theta & \le & M y \\
    & {B(:,k)}^T \Delta \theta & = & 1 \\
    & y(j) & \in & \{0,1\} \quad \forall \; j,
  \end{array}
\end{equation}
where $M$ is a constant required to be at least ${\left\|B^T\right\|}_\infty = {\left\|B\right\|}_1$ (i.e., the maximum column sum of the absolute values of the entries of $B$) \cite{SSJ_CDC2011_mincut}. Because of the binary decision variables in $y$, (\ref{opt:si_MILP}) is a mixed integer linear programming (MILP) problem \cite{BT97}. It can be solved by a standard solver such as CPLEX. The correctness of the $l_0$ approach is a direct consequence that (\ref{opt:si_MILP}) is a reformulation of (\ref{opt:card_min_con_var}). As a result, both the $l_1$ and $l_0$ approaches are guaranteed to correctly solve (\ref{opt:card_min_con_var}) by theory. Fig.~\ref{fig:security_index_l1} shows the sorted security indices (i.e., optimal objective values of (\ref{opt:card_min_con_var})) for the four larger benchmark systems. The security indices are computed using the $l_1$ approach. As a comparison, the security indices are also computed using the $l_0$ approach, and they are shown in Fig.~\ref{fig:security_index_l0}. The two figures reaffirm the theory that the proposed $l_1$ approach computes the security indices exactly. Fig.~\ref{fig:security_index_l1} (or Fig.~\ref{fig:security_index_l0}) indicates that the measurement systems are relatively insecure, as there exist many measurements with very low security indices (i.e., equal to 1 or 2).
\begin{figure}[bht]
\centering
\centering{\includegraphics[scale=0.63]{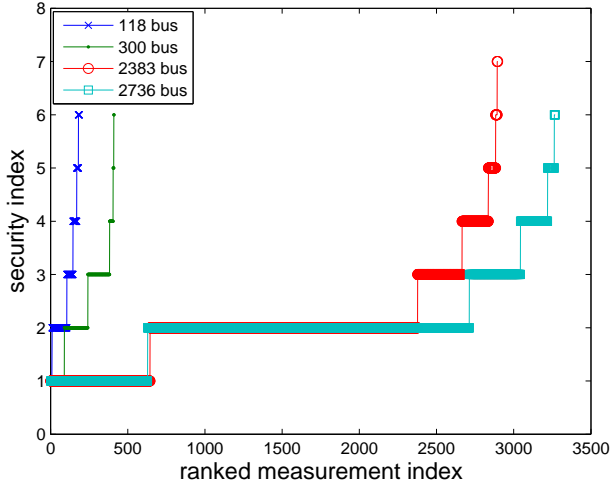}}
\caption{Security indices using the $l_1$ approach}\label{fig:security_index_l1}
\end{figure}
\begin{figure}[bht]
\centering
\centering{\includegraphics[scale=0.63]{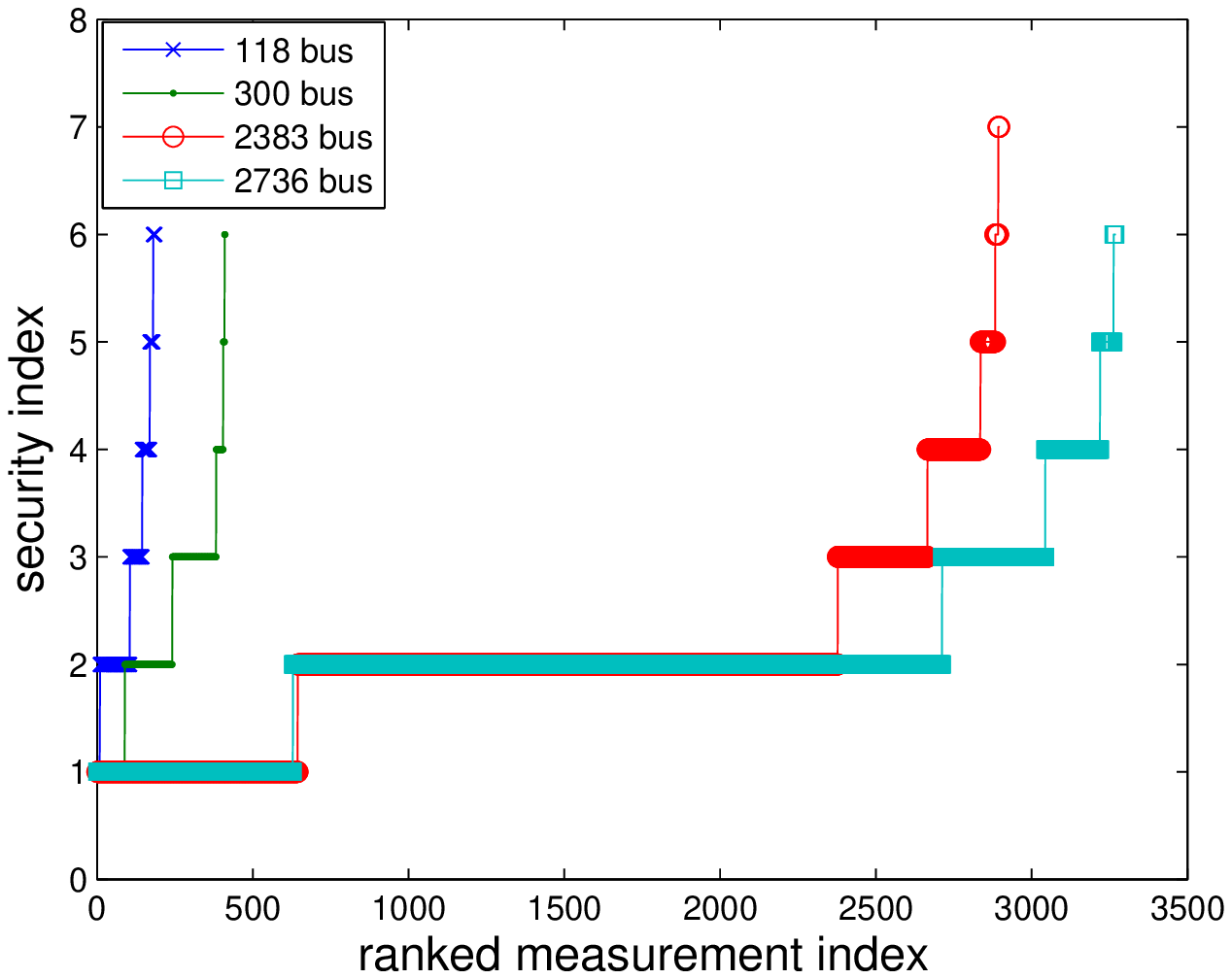}}
\caption{Security indices using the $l_0$ approach}\label{fig:security_index_l0}
\end{figure}

\begin{figure}[!tbh]
\centering
\centering{\includegraphics[scale=0.63]{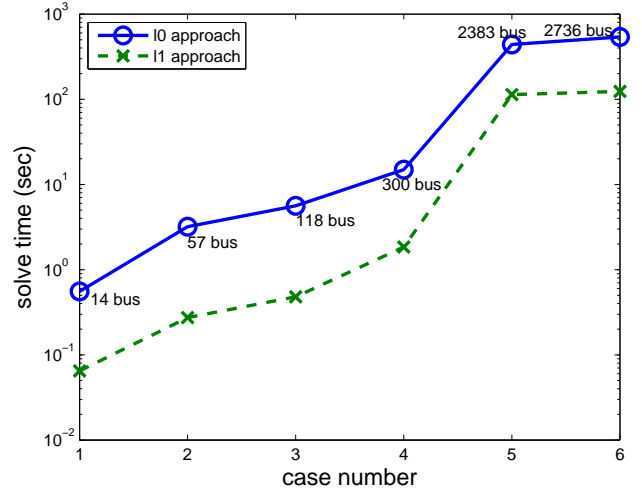}}
\caption{Solve-time for computing all security indices for different benchmark systems}\label{fig:solve_time}
\end{figure}
In terms of computation time performances, it is well-known that the $l_0$ approach is much more time-consuming than the $l_1$ approach since a MILP problem is much more difficult to solve than a LP problem of the same size \cite{BT97}. Fig.~\ref{fig:solve_time} shows the solve-time for computing all security indices for each benchmark system, using the $l_1$ and $l_0$ approaches. It verifies that the proposed $l_1$ approach is more effective. In the above illustration, all computations are performed on a dual-core Windows machine with 2.4GHz CPU and 2GB of RAM.

\section{Conclusion} \label{sec:conclusion}
The cardinality minimization problem is important but in general difficult to solve. An example is shown in this paper as the smart grid security index problem in (\ref{opt:card_min_con_protected}). The $l_1$ relaxation has demonstrated promise but to establish the cases where it provides exact solutions is non-trivial. Well-known results based on mutual coherence and RIP provide sufficient conditions under which a unique optimal solution solves both the cardinality minimization problem and its $l_1$ relaxation. However, this paper identifies a class of application motivated problems (as in (\ref{opt:l0})) which can be shown to be solvable by $l_1$ relaxation, even though results based on mutual coherence and RIP cannot make the assertion. In fact, the optimal solution to (\ref{opt:l0}) might not be unique. The key property that leads to the conclusion of this paper is total unimodularity of the constraint matrix. The total unimodularity of matrix $A$ in (\ref{opt:l0}) leads to two important consequences. (\ref{opt:l0}) is equivalent to its $\infty$-norm restricted version in (\ref{opt:l0_NB}). Furthermore, (\ref{opt:l0_NB}) can be solved exactly by solving the LP problem in (\ref{opt:l1_LP}), thus establishing the conclusion that $l_1$ relaxation exactly solves (\ref{opt:l0}).


\section*{Appendix}

\subsection{Proof of the equivalence between (\ref{opt:card_min_con_protected}) and (\ref{opt:card_min_con_var})}
Note that the constraint $H(\mathcal{I},:) \Delta \theta = 0$ implies that ${\left\| H \Delta \theta \right\|}_0 = {\left\| H(\bar{\mathcal{I}},:) \Delta \theta \right\|}_0$. Since $P$ consists of rows of an identity matrix and $D$ is diagonal and nonsingular, for all $\mathcal{J} \subset \{1,2,\ldots,m\}$, there exists a diagonal and nonsingular matrix $D_{\mathcal{J}}$ such that $P(\mathcal{J},:) D = D_{\mathcal{J}} P(\mathcal{J},:)$. In particular, let $D_{kk}$ be a positive scalar such that $P(k,:) D = D_{kk} P(k,:) = P(k,:) D_{kk}$. The above implies that for all $\Delta \theta$,
\begin{displaymath}
\begin{array}{l}
  P(k,:) B^T \Delta \theta = 1 \; \textrm{if and only if} \vspace{1mm} \\
  P(k,:) D_{kk} B^T ({D_{dd}}^{-1} \Delta \theta) = P(k,:) D B^T ({D_{dd}}^{-1} \Delta \theta) = 1.
\end{array}
\end{displaymath}
In addition, for all $\Delta \theta$
\begin{displaymath}
  \begin{array}{l}
    P(\mathcal{I},:) B^T \Delta \theta = 0 \; \textrm{if and only if} \vspace{1mm} \\
    {D_{kk}}^{-1} D_{\mathcal{I}} P(\mathcal{I},:) B^T \Delta \theta = P(\mathcal{I},:) D B^T ({D_{dd}}^{-1} \Delta \theta) = 0.
  \end{array}
\end{displaymath}
Finally, for all $\Delta \theta$
\begin{displaymath}
  \begin{array}{rcl}
    {\left\| P(\bar{\mathcal{I}},:) B^T \Delta \theta \right\|}_0 &
    = & {\left\| {D_{kk}}^{-1} D_{\bar{\mathcal{I}}} P(\bar{\mathcal{I}},:) B^T \Delta \theta \right\|}_0 \vspace{1mm} \\
    & = & {\left\| P(\bar{\mathcal{I}},:) D B^T ({D_{dd}}^{-1} \Delta \theta) \right\|}_0.
  \end{array}
\end{displaymath}
Applying the definition of $H$ in (\ref{def:H_matrix}) and a change of decision variable to ${D_{kk}}^{-1} \Delta \theta$ shows that (\ref{opt:card_min_con_protected}) and (\ref{opt:card_min_con_var}) are equivalent.


\subsection{Proof of Proposition~\ref{thm:con_obsv_eq}}
Part (a) is trivial.

For the necessary part of (b), condition I is necessary because if $H(k,:) = 0$ then ${\rm rank}(H(\bar{\mathcal{J}},:)) = {\rm rank}(H(\bar{\mathcal{J}}\cup\{k\},:))$ for all $\mathcal{J}$ (meaning that (\ref{opt:card_min_obsv}) is infeasible). Condition II is also necessary because if ${\rm rank}(H) < n$. then there does not exist any $\mathcal{J}$ such that ${\rm rank}(H(\bar{\mathcal{J}},:)) = n$.

For the sufficiency part of (b), assume that conditions I and II are satisfied. Then by part (a) problem~(\ref{opt:card_min_con}) is feasible. Hence it has an optimal solution denoted as $\theta^\star$. Define ${\mathcal{J}}_{\theta^\star} \subset \{1,2,\ldots,m\}$ such that $p \in {\mathcal{J}}_{\theta^\star}$ if and only if $H(p,:) \theta^\star \neq 0$. By definition of ${\mathcal{J}}_{\theta^\star}$, ${\rm rank}(H(\overline{{\mathcal{J}}_{\theta^\star}},:)) < n$. Also, $k \in {\mathcal{J}}_{\theta^\star}$ because $H(k,:) \theta^\star = 1$. If ${\rm rank}(H(\overline{{\mathcal{J}}_{\theta^\star}} \cup \{k\},:)) = n$, then ${\mathcal{J}}_{\theta^\star}$ is feasible to (\ref{opt:card_min_obsv}), thus showing that (\ref{opt:card_min_obsv}) is feasible. To show this, first consider the case when $\left\|H \theta^\star\right\|_0 = 1$. Then ${\mathcal{J}}_{\theta^\star} = \{k\}$ and ${\rm rank}(H(\overline{{\mathcal{J}}_{\theta^\star}} \cup \{k\},:)) = {\rm rank}(H) = n$ because of condition II (i.e., $H$ has full column rank). Next consider the case when $\left\|H \theta^\star\right\|_0 > 1$ (i.e., $\left|{\mathcal{J}}_{\theta^\star} \setminus \{k\}\right| > 0$). If ${\rm rank}(H(\overline{{\mathcal{J}}_{\theta^\star}} \cup \{k\},:)) < n$, then there exists $\tilde{\theta} \neq 0$ such that $H(\overline{{\mathcal{J}}_{\theta^\star}} \cup \{k\},:) \tilde{\theta} = 0$. In particular, $H(k,:) \tilde{\theta} = 0$. Also, condition II implies that $H({\mathcal{J}}_{\theta^\star} \setminus \{k\},:) \tilde{\theta} \neq 0$ (since otherwise $H \tilde{\theta} = 0$). Let $q \in {\mathcal{J}}_{\theta^\star} \setminus \{k\}$ such that $H(q,:) \tilde{\theta} \neq 0$. Note also that by definition of ${\mathcal{J}}_{\theta^\star}$, $H(q,:) \theta^\star \neq 0$. Construct $\theta^\prime \triangleq (H(q,:) \tilde{\theta}) \theta^\star - (H(q,:) \theta^\star) \tilde{\theta}$. Then $H(k,:) \theta^\prime = 1$, $H(p,:) \theta^\prime = 0$ whenever $H(p,:) \theta^\star = 0$, but $H(q,:) \theta^\prime = 0$ while $H(q,:) \theta^\star \neq 0$. This implies that $\theta^\prime$ is feasible to (\ref{opt:card_min_con}) with a strictly less objective value than that of $\theta^\star$, contradicting the optimality of $\theta^\star$. Therefore, the claim that ${\rm rank}(H(\overline{{\mathcal{J}}_{\theta^\star}} \cup \{k\},:)) = n$ is true. This implies that ${\mathcal{J}}_{\theta^\star}$ is feasible to (\ref{opt:card_min_obsv}), establishing the sufficiency of part (b).

For part (c), under conditions I and II both (\ref{opt:card_min_con}) and (\ref{opt:card_min_obsv}) are feasible. In addition, ${\mathcal{J}}_{\theta^\star}$ constructed in the proof of the sufficiency part of (b) satisfies $\left|{\mathcal{J}}_{\theta^\star}\right| = \left\|H \theta^\star\right\|_0$, for $\theta^\star$ being an optimal solution to (\ref{opt:card_min_con}). This means that the optimal objective function value of (\ref{opt:card_min_obsv}) is less than or equal to that of (\ref{opt:card_min_con}). For the converse, suppose that ${\mathcal{J}}^\star$ is optimal to (\ref{opt:card_min_obsv}), then the feasibility of ${\mathcal{J}}^\star$ implies that there exists $\theta_{{\mathcal{J}}^\star} \neq 0$ such that $H(\bar{{\mathcal{J}}^\star},:) \theta_{{\mathcal{J}}^\star} = 0$. This also implies that $\left\|H \theta_{{\mathcal{J}}^\star}\right\|_0 \le \left|{\mathcal{J}}^\star\right|$. If $H(k,:) \theta_{{\mathcal{J}}^\star} = 0$, then $H(\bar{{\mathcal{J}}^\star} \cup \{k\},:) \theta_{{\mathcal{J}}^\star} = 0$. This implies that ${\rm rank} (H(\bar{{\mathcal{J}}^\star} \cup \{k\},:)) < n$, contradicting the feasibility of ${\mathcal{J}}^\star$. Therefore, there exists a scalar $\alpha$ such that $H(k,:) (\alpha \theta_{{\mathcal{J}}^\star}) = 1$. Consequently, $\alpha \theta_{{\mathcal{J}}^\star}$ is feasible to (\ref{opt:card_min_con}) with an objective function value less than or equal to the optimal objective function value of (\ref{opt:card_min_obsv}).

\bibliographystyle{IEEEtran}
\bibliography{exact_l1,gmc}

\end{document}